\documentclass[11pt,reqno]{amsart}
\usepackage{amsmath, amssymb, amsthm} 
\usepackage[pagewise]{lineno}\linenumbers
\nolinenumbers
\usepackage{mathrsfs}
\usepackage[breaklinks]{hyperref}
\usepackage{tikz}
\usepackage{mathrsfs}
\usepackage{multirow}
\usepackage{latexsym}
\usepackage{upref, eucal}
\usepackage{pgfplots}
\usepackage{rotating, tikz}
\usetikzlibrary{matrix,arrows,backgrounds}
\usepackage[headheight=110pt,top=1in, bottom=.9in, left=0.8in, right=0.8in]{geometry}

\usepackage{url}
\usepackage{amssymb}

\newcommand{\df}{\dfrac}
\newcommand{\tf}{\tfrac}

\newcommand{\Res}{\operatorname{Res}}
\renewcommand{\Re}{\operatorname{Re}}

\renewcommand{\(}{\left\(}
\renewcommand{\)}{\right\)}
\renewcommand{\[}{\left\[}
\renewcommand{\]}{\right\]}
\renewcommand{\i}{\infty}
\numberwithin{equation}{section}
\theoremstyle{plain}
\newtheorem{theorem}{Theorem}[section]

\newtheorem{conjecture}{Conjecture}[section]
\newtheorem{corollary}[theorem]{Corollary}
\newtheorem{proposition}[theorem]{Proposition}

\newtheorem{example}[]{Example}
\newtheorem{remark}[]{Remark}

\makeatletter
\def\proof{\@ifnextchar[{\@oproof}{\@nproof}}
\def\@oproof[#1][#2]{\trivlist\item[\hskip\labelsep\textit{#2 Proof of\
		#1.}~]\ignorespaces}
\def\@nproof{\trivlist\item[\hskip\labelsep\textit{Proof.}~]\ignorespaces}

\makeatother

\makeatletter
\def\@tocline#1#2#3#4#5#6#7{\relax
	\ifnum #1>\c@tocdepth 
	\else
	\par \addpenalty\@secpenalty\addvspace{#2}%
	\begingroup \hyphenpenalty\@M
	\@ifempty{#4}{%
		\@tempdima\csname r@tocindent\number#1\endcsname\relax
	}{%
		\@tempdima#4\relax
	}%
	\parindent\z@ \leftskip#3\relax \advance\leftskip\@tempdima\relax
	\rightskip\@pnumwidth plus4em \parfillskip-\@pnumwidth
	#5\leavevmode\hskip-\@tempdima
	\ifcase #1
	\or\or \hskip 1em \or \hskip 2em \else \hskip 3em \fi%
	#6\nobreak\relax
	\dotfill\hbox to\@pnumwidth{\@tocpagenum{#7}}\par
	\nobreak
	\endgroup
	\fi}
\makeatother

\allowdisplaybreaks
\begin{document}
	\title[]{Equivalence criteria for the two--term functional equations for Herglotz--Zagier type functions}
	\author{Sumukha Sathyanarayana}
	\author{N. Guru Sharan }
	\address{Department of Mathematics, Central University of Karnataka, Kadaganchi, Kalaburagi, Karnataka-585367, India.}
	\email{sumukhas@cuk.ac.in, neerugarsumukha@gmail.com}
	\address{Department of Mathematics, Indian Institute of Science, Bangalore 560012, India.}
	\email{ngurusharan@iisc.ac.in, sharanguru5@gmail.com}

	\begin{abstract}	
For any integer $a$ and non-negative integer $b$, we define a Herglotz--Zagier (HZ) type function $F_{a,b}(x)$ by an absolutely convergent series involving the Digamma function $\psi(x)$. For each such $F_{a,b}(x)$, we associate an integer weight. In the literature,  Ramanujan, Guinand, Zagier, Vlasenko-Zagier have derived two-term functional equations for some HZ type functions of positive weights. In this paper, we study a class of HZ type function associated with negative weights, and obtain their two-term functional equations. Parallelly, we associate an integer weight to the Kronecker limit type formula for the generalized Mordell--Tornheim zeta function $\Theta(r,s,t,x)$. We establish that any two-term functional equation for HZ type function is equivalent to a Kronecker limit type formula of $\Theta(r,s,t,x)$, preserving weight. As a consequence, we derive new Kronecker limit type formulas and obtain a new special value of the Mordell--Tornheim zeta function $\zeta_{\textup{MT}}(r,s,t)$.  We also obtain results of Ramanujan, Guinand, Zagier, and Vlasenko-Zagier as consequences, to show that the Mordell--Tornheim zeta function lies centrally between many known modular relations.
	\end{abstract}
	\maketitle
	\section{Introduction}\label{introduction}



Over the century, several prominent mathematicians such as Hecke \cite{hecke}, Herglotz \cite{herglotz}, Ramachandra \cite{rama}, Stark \cite{stark}, and Zagier \cite{zagier} gave important contributions towards Kronecker limit formulas, after Kronecker himself, whose results were concerned with the imaginary quadratic fields. Siegel \cite{siegel} connected the Kronecker limit formula with the theory of Dirichlet $L$-functions. 
Zagier \cite{zagier} then studied the Kronecker limit formula for real quadratic fields: for a real quadratic field $K=\mathbb{Q}(\sqrt{D})$ of discriminant $D>0$, and a narrow ideal class $\mathscr{B}$, the $\zeta$-function associated to $\mathscr{B}$ is defined for Re$(s)>1$ by
\begin{align*}
	\zeta(s,\mathscr{B}) := \sum_{\mathfrak{a}\in \mathscr{B}} \frac{1}{N(\mathfrak{a})^s},
\end{align*}
where the summation is over all integral ideals in the class $\mathscr{B}$. He proved the following formula: 
\begin{align}\label{KLFzagier}
	\lim_{s \to 1} \left( D^{s/2} \zeta(s,\mathscr{B}) - \frac{\log \epsilon}{s-1} \right) = \sum_{w} P(w,w')
\end{align}
where $\epsilon>1$, is the smallest unit of $K$ of norm $1$, where the summation on the right-hand side of \eqref{KLFzagier} is over all the elements in the set of larger roots of the equation $cw^2 + bw +a =0$ associated to each reduced primitive quadratic form $Q(x,y) =ax^2 +bx y+cy^2$ with discriminant $D= b^2-4ac$, belonging to the class $\mathscr{B}$ , i.e., $w=\frac{1}{2a}(-b+\sqrt{D})$. Note that $w'$ represents the smaller roots, i.., $w'<w$. Here, 
 $P(x,y)$ is a function of two variables $x>y>0$, defined by,
\begin{align*}
	P(x,y):= F(x)-F(y)+\textup{Li}_2\left( \frac{y}{x} \right) -\frac{\pi^2}{6} + \log\left( \frac{x}{y} \right) \left( \gamma -\frac{1}{2} \log(x-y) + \frac{1}{4} \log \left(\frac{x}{y}\right) \right).
\end{align*}
Here, $\gamma$ is Euler's constant, $\textup{Li}_2$ is the dilogarithm function and $F(x)$ is defined by the following infinite series, 
\begin{equation}\label{herglotzdef}
	F(x):=\sum_{n=1}^{\infty}\frac{\psi(nx)-\log(nx)}{n},
\end{equation} 
where $\psi(x)=\frac{\Gamma'\hspace{-0.05 cm}(x)}{\Gamma(x)}$, is the digamma function, i.e., the logarithmic derivative of the Gamma function  $\Gamma(x)$. The function $F(x)$ was studied for the first time by Herglotz \cite{herglotz}, thus, is called as the \textit{Herglotz--Zagier function}. $F(x)$ can be analytically extended to the slit complex plane $x\in\mathbb{C}\backslash(-\infty,0]$. From the asymptotic expansion of $\psi(x)$ as $x\to \infty$, 
\cite[p.~259, formula 6.3.18]{Handbook},
one can see that that the series in \eqref{herglotzdef} converges absolutely for any $x\in\mathbb{C}\backslash(-\infty,0]$. Zagier \cite[Equations (7.4) and (7.8)]{zagier} proved the following two- and three-term functional equations for $F(x)$:
\begin{align}
	&F(x)+F\left(\frac{1}{x}\right)=2F(1)+\frac{1}{2}\log^{2}(x)-\frac{\pi^2}{6x}(x-1)^2,\label{fe2}\\
	&F(x)-F(x+1)-F\left(\frac{x}{x+1}\right)=-F(1)+\textup{Li}_2\left(\frac{1}{1+x}\right),\label{fe1}
\end{align}
with \cite[Equation (7.12)]{zagier}
\begin{equation}\label{ef1}
	F(1)=-\frac{1}{2}\gamma^2-\frac{\pi^2}{12}-\gamma_1,
\end{equation}
where $\gamma$ is Euler's constant and $\gamma_1$ is the first Stieltjes constant. 

In \cite{vz}, Vlasenko and Zagier studied the higher Kronecker ``limit'' formula analogous to \eqref{KLFzagier} for integers $s \geq 2$. They show that the function $\zeta(s , \mathscr{B})$ is convergent at these values, hence the quotes for limit. This formula involves the derivatives of the higher Herglotz function $F_r(x)$, which is defined by the following absolutely convergent series:
\begin{equation}\label{higherherglotz}
	F_r(x):=\sum_{n=1}^{\infty}\frac{\psi(nx)}{n^{r-1}}\hspace{5mm}\left(x\in\mathbb{C}\backslash(-\infty,0]\right).
\end{equation} 
	They further showed that $F_r(x)$ satisfies two- and three-term functional equations \cite[Proposition 4]{vz}, analogous to \eqref{fe2} and \eqref{fe1}, given by,
	\begin{align}
		&F_{r+1}(x) + (-x)^{r-1} F_{r+1}\left( \frac{1}{x} \right) = \zeta(r+1) \left( (-x)^{r} -\frac{1}{x} \right) - \sum_{\ell=1}^{r} \zeta(\ell) \zeta(r-\ell+1) (-x)^{\ell-1}, \label{vz2term} \\
		&F_{r+1}(x) - F_{r+1}(x+1) + (-x)^{r-1} F_{r+1} \left( \frac{x+1}{x} \right) = \zeta(r+1) \left( \frac{(-x)^{r}}{x+1}-\frac{1}{x} \right) - \sum_{\ell=1}^{r} \zeta_{\textup{D}}(r-\ell+1, \ell) (-x)^{\ell-1}, \label{vz3term}
	\end{align}
	where the double zeta function $\zeta_{\textup{D}}(s_1,s_2)$ is defined \cite[Section 1.1]{vz} by,
	\begin{align}
		\zeta_{\textup{D}}(s_1,s_2) = \sum_{n=1}^{\infty} \sum_{m=n+1}^{\infty} \frac{1}{m^{s_1} n^{s_2}}, \hspace{1 cm} \textup{(for $s_1\geq2$ and $s_2\geq1$).} \label{doublezetadef}
	\end{align}
	As noted in \cite[p. 28]{vz}, the divergent values $\zeta(1)$ and $\zeta(1,r)$ should be interpreted as $\gamma$ and $-(\zeta(r,1)+\zeta(r+1)-\gamma \zeta(r))$\footnote{The interpretation of $\zeta(1,r)$ given in \cite[p. 28]{vz} has a minus sign missing. The convention given here is the correct one.}, respectively, here and everywhere below.
	
Several functions similar to the Herglotz-Zagier function $F(x)$ have been studied in the literature. We list below few results due to Ramanujan and Guinand. In general, for $\Re(x)>0$, a transformation of the form $F(x)= F(1/x)$ is called a \textit{modular relation}. Equivalently, we can also write it as form $F(-1/z)=F(z)$, where $z\in\mathbb{H}$, the upper half complex plane. Ramanujan \cite{bcbad} had recorded the following interesting Modular relation, on page 220 of Lost Notebook:
for $x>0$, let us denote
\begin{equation*}
	\phi(x):=\psi(x)+\frac{1}{2x}-\log(x).
\end{equation*}
Then, for $x>0$, 
\begin{align}
	\sqrt{x}\left\{\dfrac{\gamma-\log(2\pi x)}{2 x}+\sum_{n=1}^{\i}\phi(n x)\right\}
	&=\sqrt{\frac{1}{x}}\left\{\df{x(\gamma-\log\left(\frac{2\pi}{x}\right))}{2}+\sum_{n=1}^{\i}\phi\left(\frac{n}{x}\right)\right\} \notag \\
	&=-\df{1}{\pi^{3/2}}\int_0^{\i}\left|\Xi\left(\df{1}{2}y\right)\Gamma\left(\df{-1+iy}{4}\right)\right|^2
	\df{\cos\left(\tf{1}{2}y\log x\right)}{1+y^2}\, dy,  \label{w1.26}
\end{align}
where
\begin{align*}
	\xi(s)&:=\frac{1}{2}s(s-1)\pi^{-\frac{s}{2}}\Gamma\left(\frac{s}{2}\right)\zeta(s),\\
	\Xi(y)&:=\xi\left(\frac{1}{2}+i y\right),
\end{align*}
denote the Riemann's functions, as in \cite{titch}. Dixit \cite[Theorem 4.1]{Analogues} obtained a generalization of the above result in the variable $z$, with the Hurwitz zeta function $\zeta(z,x)$ in the place of $\psi(x)$ in the first equality, the modular relation.

Guinand \cite{apg3} gave modular relations which are analogous to the first equality in \eqref{w1.26} involving the derivatives of $\psi(x)$. For $p\in\mathbb{N}$ and $p>2$, he showed that,
	\begin{align}\label{guinand2}
		x^{\frac{p}{2}}\sum_{j=1}^{\infty}\psi^{(p-1)}(1+jx)=\bigg(\frac{1}{x}\bigg)^{\frac{p}{2}}\sum_{j=1}^{\infty}\psi^{(p-1)}\bigg(1+\frac{j}{x}\bigg).
	\end{align}
Also, analogously, for the first derivative, he proved the identity,
	\begin{align}\label{guinand1}
		x\sum_{j=1}^{\infty}\left(\psi'(1+jx)-\frac{1}{jx}\right)-\frac{1}{2}\log(x)=	\frac{1}{x}\sum_{j=1}^{\infty}\left(\psi'\bigg(1+\frac{j}{x}\bigg)-\frac{x}{j}\right)-\frac{1}{2}\log\bigg(\frac{1}{x}\bigg).
	\end{align}
	
	In a general setting, we define a \textit{Herglotz--Zagier type function} $F_{a,b}(x)$ (referred to as a HZ type function henceforth)  to be the following:
	\begin{align}
		F_{a,b}(x) := \sum_{n=1}^{\infty} \frac{1}{n^a} (\psi^{(b)}(nx) + \alpha(nx)), \label{Fdefn}
	\end{align}
	where, $a$ is any integer, $b$ is any non-negative integer representing the order of derivative for $\psi(x)$, and $\alpha(x)$ is the function which is necessary to make the summand absolutely convergent. For any fixed pair $a$ and $b$, one can obtain $\alpha(x)$ explicitly by looking at the asymptotic of $\psi^{(b)}(y)$ as $y \to \infty$, for instance, written down in \cite[Lemma 8.2]{dss} and \cite[Equation (2.1)]{dss}.
	
	For each HZ type function $F_{a,b}$ as defined in \eqref{Fdefn}, let us assign an integer \textit{weight} to it as follows:
	\begin{align}
		\text{Weight}\left(F_{a,b}\right):= a+b+1. \label{weight}
	\end{align}
	In Table \ref{table}, we list a few examples of the above discussed HZ type functions with $\alpha(x)$ explicitly written down, along with their respective weights. To the best of our knowledge, all the HZ type functions thus far studied in the literature have been associated with positive weights alone. The authors, through this paper, hereby initiate the study on corresponding functions of non-positive weights. 
	\renewcommand{\arraystretch}{1.7}
	\begin{table}[h!]
		\begin{center}
			\begin{tabular}{ | c| c| c| } 
				\hline
				Herglotz--Zagier type function& Reference & Weight \\
				\hline
				\hline
				$\sum_{n=1}^{\infty}\frac{1}{n} \left( \psi(nx)-\log(nx) \right)$ & Zagier \eqref{fe2} & $2$ \\
				\hline
				$\sum_{n=1}^{\infty}\frac{1}{n^r} \left( \psi(nx) \right),\quad (r\geq2)$& Vlasenko-Zagier \eqref{vz2term}  & $r+1$ \\ 
				\hline
				$\sum_{n=1}^{\infty} \left( \psi(nx)+\frac{1}{2nx}-\log(nx) \right)$& Ramanujan \eqref{w1.26}& $1$ \\ 
				\hline
				$\sum_{n=1}^{\infty}\left(\psi^{(p-1)}(1+nx)\right),\quad (p \geq 3)$& Guinand \eqref{guinand2}& $p$ \\
				\hline 
				$\sum_{n=1}^{\infty}\left(\psi'(1+nx)-\frac{1}{nx}\right)$&Guinand \eqref{guinand1}& $2$ \\
				\hline
				$\sum_{n=1}^{\infty} n \left( \psi(nx)+\frac{1}{2nx}-\log(nx)+\frac{1}{12n^2x^2} \right)$ & New \eqref{newfe} & $0$\\ [0.5ex]
				\hline  
			\end{tabular}
			\caption{List of various Herglotz--Zagier type functions along with their weight.}
			\label{table}
		\end{center}
	\end{table}
	
Before we look at the new findings of this paper, let us denote the region
\begin{align}
	\mathscr{D}:=\big\{(r,s,t) \in \mathbb{C}^3 \ \big| \Re(r+t)>1, \Re(s+t)>1 , \Re(r+s+t)>2\big\}. \label{Dregion}
\end{align}
In the region $\mathscr{D},$ the Mordell--Tornheim zeta function $\zeta_{\textup{MT}} (r,s,t)$, as named in \cite{matsumoto-bonn}, is defined by the absolutely convergent series  
\begin{align}
	\zeta_{\textup{MT}} (r,s,t):=\sum_{n=1}^{\infty} \sum_{m=1}^{\infty} \frac{1}{n^{r} m^{s} (n+m)^{t}}.
\end{align}
Matsumoto \cite[Theorem 1]{matsumoto} used the Mellin-Barnes method to show that $\zeta_{\textup{MT}}(r,s,t)$ can be meromorphically continued to the entire $\mathbb{C}^3$ space with all its singularities lying on the hyperplanes of $\mathbb{C}^3$ defined by any of the following equations,
\begin{equation} \label{pole conditions}
	\left.
	\hspace{5 cm}
	\begin{aligned}
		&r+t=1-\ell, \\
		&s+t=1-\ell, \\
		&r+s+t=2,\\
	\end{aligned}
	\hspace{4 cm} \right\}
\end{equation}
for any $\ell \in \mathbb{N}\cup\{0\}$. 

The behavior of $\zeta_{\textup{MT}} (r,s,t)$ around its singularities is largely unexplored. In a recent work with Dixit \cite{dss2}, the authors considered its generalization
\begin{align}
	\Theta(r,s,t,x):=\sum_{n=1}^{\infty}  \sum_{m=1}^{\infty}\frac{1}{n^{r}m^{s}(n+mx)^{t}}, \label{Tdef}
\end{align}
for $x>0$ and $(r,s,t) \in \mathscr{D}$ as defined in \eqref{Dregion}. Clearly, $\Theta(r,s,t,1) =\zeta_{\textup{MT}} (r,s,t)$. We can use the Mellin-Barnes method to obtain the meromorphic continuation of $\Theta(r,s,t,x)$ to the entire $\mathbb{C}^3$ and show that its singularities is identical to that of $\zeta_{\textup{MT}}(r,s,t)$ as mentioned in \eqref{pole conditions}. One can easily observe that $\Theta(r,s,t,x)$ satisfies the following properties for $x>0$ and $(r,s,t) \in \mathscr{D}$, which we call \textit{split} and \textit{inversion}, respectively:
\begin{align}
	&\Theta(r,s,t,x) =  \Theta(r-1,s,t+1,x)+ x\Theta(r,s-1,t+1,x), \label{Tsplit}  \\
	&	\Theta(r,s,t,x)= x^{-t} \Theta(s,r,t,\tfrac{1}{x}).  \label{Tinv}
\end{align} 

As noted in the beautiful survey article by Zagier \cite[Section 8]{zagiersurvey}, we can define a weight to $\zeta_{\textup{MT}}(r,s,t)$, for each fixed tuple, to be the value $r+s+t$. We extend the notion and define the weight of $\Theta(r,s,t,x)$ to be the value $r+s+t$, which remains independent of $x$. 

One of the goals of this paper is to develop a broader understanding of the aforementioned two-term functional equations \eqref{w1.26}, \eqref{guinand2}, \eqref{guinand1}, \eqref{fe2}, and \eqref{vz2term}, through the properties of $\Theta(r,s,t,x)$, which facilitates us to come up a broad and unified method to understand these results. In Section \ref{SectionMTPer}, we explain the phenomena in-detail and list the observations in Table \ref{table2}, for reader's ease.

In \cite[Theorem 1.1]{dss2}, it was shown that, as $t \to 0$,
\begin{align}\label{principal part equation}
	\Theta(1,1,t,x)=\frac{2}{t^2}+\frac{2\gamma-\log(x)}{t}+\gamma^2-\gamma\log(x)-\frac{\pi^2}{6}+O_{x}(|t|).
\end{align}
Results of this kind did not exist in the literature prior to this, even with $x=1$. 
The authors have further showed the following Kronecker limit type formula for $\Theta(r,r,t,x)$, for any integer $r>1$, in the third variable $t$, around $t=1-r$, analogous to \eqref{principal part equation}, in \cite[Equation (2.25)]{ss2}:
	\begin{align}\label{KLF-r}
		\Theta(r,r,t,x) &= \frac{\zeta(r)\left(1 + x^{r-1} \right)}{(t-(1-r))}  +  x^{r-1} \zeta(r)(\gamma-\log(x)) + \gamma \zeta(r)   \notag \\
	&\quad  + x^{r-1} \sum_{\substack{k=1}}^{r-2} \binom{r-1}{k}  \frac{1}{ x^{k}} \zeta(r-k)\zeta(k+1)   + O_{x}(|t-(1-r)|).
	\end{align}
The singularities considered for the above Laurent series are of type $r+t=1$ alone, where $\Theta(r,r,t,x)$ has a first order pole, unlike in \eqref{principal part equation}, where it has a second order pole and the singularity is due to all the three conditions mentioned in \eqref{pole conditions}. Hence, it is clear that Theorem \ref{KLF-r} is not a generalization of \eqref{principal part equation}, but an analogue.

We can further extend the notion of the weight to each Kronecker limit type formula for $\Theta(r,s,t,x)$, to be defined by the sum $r+s+t$ of the coordinates around which the formula is defined. For instance, the weight of \eqref{principal part equation} is $2$, and the weight of \eqref{KLF-r} is $r+1$. 
 
	For $\textup{Re}(z)>0, z\neq1,$ and $x\in\mathbb{C}\backslash(-\infty,0]$, the authors in \cite{dss2} have defined the following \textit{generalized Herglotz function} by
\begin{align*}
	\Phi(z, x):=\sum_{n=1}^{\infty} \frac{1}{n}\left( \zeta(z,nx)- \frac{(nx)^{1-z}}{z-1} \right),
\end{align*}	 
where $\zeta(z, x)$ is the well-known Hurwitz zeta function. In \cite[Theorem 1.3]{dss2}, the following decomposition identity of $\Theta(r,s,t,x)$ has been proved:
	for \textup{Re}$(z)>1$ and $x>0$,
	\begin{align}\label{decomposition}
		\Phi(z, x)+x^{1-z}\Phi\left(z, \frac{1}{x}\right)= \Theta(1,1,z-1,x) - \frac{(1+x^{1-z})}{z-1}\zeta(z)+(x+x^{-z})\zeta(z+1).
	\end{align}
It can be seen that $\lim_{z\to 1} \Phi(z,x) = -F(x)$ as shown in \cite[Lemma 3.2]{dss2}. Also, letting $z \to 1$ in \eqref{decomposition} and making use of \eqref{principal part equation} leads us to \eqref{fe2}. This hints towards new Kronecker limit formulas, a generalization each of \eqref{KLFzagier} and \cite[Theorem 2]{vz}, with $\Phi(z,x)$ appearing in the summand on the right-hand side. Due to the decomposition identity \eqref{decomposition}, the properties of $\Phi(z,x)$ can be obtained through that of $\Theta(1,1,z-1,x)$.
	
As the main result of this paper, we essentially obtain the following \textit{equivalence criteria} for the two-term functional equations of HZ type functions $F_{a,b}(x)$ with $a\ne 0$. The class $F_{0,b}(x)$ is dealt with separately in Section \ref{SectionMTPer}. 

\begin{theorem}\label{MAIN}
	For any fixed integer $k$, the two-term functional equation (or the modular relation) satisfied by a Herglotz--Zagier type function of weight $k$ is equivalent to a Kronecker limit type formula of the generalized Mordell--Tornheim zeta function $\Theta(r,s,t,x)$, of weight $k$.
\end{theorem}

\begin{remark}
	\emph{	Bochner \cite{Bochner} established the equivalence between a Hecke-type functional equation and a modular relation of certain $L$-functions, a perspective expanded by Chandrasekharan and Narasimhan \cite{KC}. In parallel, Theorem \ref{MAIN} shows that the two-term functional equation for HZ functions, which is a modular-type relation, is equivalent to the Kronecker limit formulas for the Mordell--Tornheim zeta function. The key difference lies in the analytic side: the classical framework features a functional equation, whereas here it is a limit formula capturing the principal part of the Laurent series expansion, perhaps, along with the constant term.}
\end{remark}

We prove Theorem \ref{MAIN} in Section \ref{Section3}. Since Theorem \ref{MAIN} is \textit{existential} in nature, and does not describe the nature of either the two-term functional equation or the Kronecker limit type formula, we make it more explicit through the below listed five illustrations:  

\begin{example}\label{ex1} \rm  Weight $2$ invariant: the two-term functional equation \eqref{fe2} satisfied by $F(x)$ is equivalent to the Kronecker limit type formula \eqref{principal part equation} for $\Theta(1,1,t,x)$ around $t=0$.
\end{example}

\begin{example}\label{ex2}  \rm
	Weight $r+1$ invariant: for any integer $r>1$, the two-term functional equation \eqref{vz2term} satisfied by $F_{r+1}(x)$ is equivalent to the Kronecker limit type formula \eqref{KLF-r} for $\Theta(r,r,t,x)$ around $t=1-r$.
\end{example}

\begin{example}\label{ex3}  \rm
	Weight $0$ invariant: the two-term functional equation \eqref{newfe} satisfied by $F_{-1,0}(x)$ is equivalent to the Kronecker limit type formula \eqref{from SS2} for $\Theta(0,0,t,x)$ around $t = 0$. 
\end{example}

\begin{example}\label{ex4} \rm
	Weight $3$ invariant: the two-term functional equation satisfied by $F_{-1,3}(x)$ is equivalent to the Kronecker limit type formula for $\Theta(0,0,t,x)$ around $t=3$. 
\end{example}

\begin{example}\label{ex5} \rm 
	Weight $-2$ invariant: the two-term functional equation satisfied by $F_{-3,0}(x)$ is equivalent to the Kronecker limit type formula for $\Theta(-2,-1,t,x)+ x^{1-t} \Theta\left(-2,-1, t,\frac{1}{x}\right)$ around $t=1$.
\end{example}
\noindent
We highlight below a few key observations from the above examples:
\begin{itemize}
	\item 
	In Examples \ref{ex1} and \ref{ex2}, both the functional equation of the chosen HZ type function along with the associated Kronecker limit type formula have been known before, where the equivalence alone is the new observation. In Examples \ref{ex3}, \ref{ex4}, and \ref{ex5}, the functional equations did not exist in the literature before. In Section \ref{Sec2}, we prove a new class of such two-term functional equations for HZ type functions of weight $\leq 1$. 
	
	\item Also, another stark difference between the three instances is the nature of singularities: while a second order pole and a first order pole features in Example \ref{ex1} and \ref{ex2}, respectively, but a special value features in Examples \ref{ex3} and \ref{ex5} (refer to \cite[p. 5]{ss2} for more details). Example \ref{ex4} features in the analytic region of $\Theta(r,s,t,x)$.
	
	\item The Kronecker limit type formula involves a combination of $\Theta(r,s,t,x)$ in the Example \ref{ex5} alone. The Example \ref{ex4} is the odd one containing a HZ type function with $b \ne 0$.
\end{itemize}
We explicitly describe the above listed examples in Section \ref{Section3}. As an application of the equivalence: we obtain a new functional equation using a known special value in Corollary \ref{newFEc}, and, we conversely also obtain a new special value using a new functional equation proved in Theorem \ref{newFE}, in Corollary \ref{newSP}. In Section \ref{SectionMTPer}, we tabulate the connections between the modular relations discovered by Guinand, Ramanujan, Herglotz, Zagier and Vlasenko-Zagier, and thereby show that the theory of the Mordell--Tornheim zeta function lies centrally between them.  Theorem \ref{MAIN} can be seen as the big picture behind these identities.

\subsection{The space of Herglotz--Zagier type functions}
Before we discuss further developments on the space of HZ type functions, we define the Multiple zeta values. The multi-summand
\begin{align*}
	\zeta(\textbf{s}) := \zeta(s_1,s_2, \ldots, s_{r}) = \sum_{n_1 >n_2 > \cdots >n_{r} \geq 1} \frac{1}{n_1^{s_1} n_2^{s_2} \cdots n_{r}^{s_{r}}}
\end{align*}
is convergent for $s_{i} \in \mathbb{N}$ and $s_1\geq2$. These admissible special values of $\zeta(\textbf{s})$ are called \textit{Multiple zeta values} (referred to as MZVs henceforth). The MZV $\zeta(\textbf{s})$ is assigned a \textit{weight} $|\textbf{s}|= s_1 +s_2 + \cdots + s_{r}$ and a \textit{depth} $r$. Refer to \cite[Section 1.2]{GilFresan} for finer details. Let us define \textbf{Z} to be the $\mathbb{Q}$-vector space generated by all MZVs, i.e.,
\begin{align*}
	\textup{\textbf{Z}} := \langle 1, \zeta(2), \zeta(3), \zeta(2,1), \zeta(4), \ldots \rangle_{\mathbb{Q}}.
\end{align*} 
Also, let $\textbf{Z}_k\subsetneq\textbf{Z}$ be the subspace spanned by the MZVs of weight $k$. Trivially, $\textbf{Z}_{0}=\mathbb{Q}$ and $\textbf{Z}_{1}=\{0\}$. It can be easily seen that few different MZVs satisfy algebraic relations. It can further be shown that all the relations between MZVs known so far are weight-preserving (or \textit{homogeneous}), as explained in detail \cite{zudilin}. The above fact leads us towards a conjecture of Zagier \cite{zagierannals}: it is unknown if the weight defines a \textit{grading} on $\textbf{Z}$, i.e., $	\textup{$\textbf{Z}$} \stackrel{?}{=} \bigoplus_{k\geq0} \textup{$\textbf{Z}_k$}$. 

For each $k$, there are only finitely many MZVs of weight $k$, and hence $\textbf{Z}_k$ is finite dimensional. There are many relations between MZVs (of a fixed weight). For instance, we have the \textit{shuffle relations} \cite[Section 1.3.2]{GilFresan}:
\begin{align*}
	\zeta(j) \zeta(k-j) = \sum_{\ell=2}^{k-1} \left( \binom{\ell-1}{j-1} + \binom{\ell-1}{k-j-1} \right) \zeta(\ell , k-\ell). 
\end{align*} 
It is hence conjectured \cite{zagierannals} that dim$_{\mathbb{Q}}$ $\textbf{Z}_k \stackrel{?}{=} d_k$ (defined recursively by $d_k = d_{k-2}+d_{k-3}$, along with the initial conditions $d_0=1$ and $d_1=0$).

Let $\mathbb{Q}((x))$ be the field of formal Laurent series over $\mathbb{Q}$. Let $\textbf{F}$ be the vector space spanned by all HZ type functions over $\mathbb{Q}((x))$. For any fixed integer $k$, let $\textbf{F}_k$ be the subspace of $\textbf{F}$, spanned by HZ type functions of weight $k$.

\begin{conjecture}\label{Conj1.1}
	The subspaces \textup{$\textbf{F}_k$} are in direct sum, i.e., \textup{$\textbf{F}$} is a graded linear space,
	\begin{align*}
		\textup{$\textbf{F}$} = \bigoplus_{k \in \mathbb{Z}} \textup{$\textbf{F}_k$}.
	\end{align*}
\end{conjecture}

\begin{remark}
	\rm The above mentioned Conjecture \ref{Conj1.1} on $\textbf{F}$ should be seen analogous to that on $\textbf{Z}$ \cite{zagierannals}. Unlike $\textbf{Z}_k$, the weight here can be negative integers too, i.e., $\textbf{F}$ is expected to be a direct sum over all integers. 
\end{remark}

Analogous to the MZVs, the HZ type functions (of a fixed weight) also satisfy many relations. For instance, as given in \cite[Equation 1.7]{ss2}:
\begin{align*}
	-\sum_{m=1}^{\infty} \frac{\psi'(mx+1)}{m^2} + 	\sum_{m=1}^{\infty} \frac{\psi'(\frac{m}{x}+1)}{m^2} + \frac{2}{x} \sum_{m=1}^{\infty} \frac{\gamma+\psi(mx+1)}{m^3} = \zeta^2(2).
\end{align*} 
While it is clear that $\textbf{Z}_k$ is finitely spanned, on the other hand, one can see that $\textbf{F}_k$ is a span of infinitely many HZ type functions of weight $k$, which are likely to be linearly independent over $\mathbb{Q}((x))$. Since, as shown in \cite[Theorem 3.3]{ss2}, we can obtain infinitely many relations between HZ type functions (of a fixed weight), one can still expect dim$_{\mathbb{Q}((x))}$ $\textbf{F}_k$ to be finite. Thus, it will be interesting to study dim$_{\mathbb{Q}((x))} \textbf{F}_k$, analogous to that of dim$_{\mathbb{Q}}  \textbf{Z}_k$. A study of the space $\textbf{F}_k$ could hint towards an alternative way to learn about $\textbf{Z}_k$.


\section{New functional equations}\label{Sec2}

In this section, we prove new two-term functional equations for a class of HZ type functions of non-positive integer weights, using contour integrals.

\begin{theorem}\label{newFE}
	Let $k$ be any non-negative integer. Consider the HZ type function $F_{-k,0}(x)$ of weight $1-k$ defined explicitly by: 
	\begin{align}\label{lowerherglotz}
		F_{-k,0}(x)&:= \sum_{
			n=1}^{\infty} n^k \left(\psi(nx)-\log(n x)+\frac{1}{2nx}+ \sum_{m=1}^{\lfloor \frac{k+1}{2} \rfloor } \frac{ B_{2m}}{2 m (nx)^{2m}}\right).
	\end{align}
	It satisfies the following two-term functional equation:
		\begin{align}\label{sym}
			x^{\frac{k+1}{2}} \left( F_{-k,0}(x) +\frac{(-1)^k }{2x^{k+1}}  \sum_{m=0}^{k+1} P_{m,-k}\left(x \right)\right) =	(-1)^k \left(\frac{1}{x}\right)^{\frac{k+1}{2}} \left( F_{-k,0}\left(\frac{1}{x}\right)  +\frac{(-1)^k x^{k+1}}{2}  \sum_{m=0}^{k+1} P_{m,-k}\left(\frac{1}{x}\right) \right),
		\end{align}
where
\begin{align}\label{residualterm}
P_{m,-k}(x):=\begin{cases} 
	-\gamma \zeta(-k) + \log(x)  \zeta(-k) + \zeta'(-k),	 & \textup{if}~ m=0,\\
		(-1)^{m+1} \zeta(m-k)\zeta(1-m) x^m,  & \textup{if}~ 1\le m \le k, \\
	(-1)^k x^{k+1} (\gamma \zeta(-k)+ \log(x) \zeta(-k)- \zeta'(-k)),	 & \textup{if}~ m=k+1 .
	\end{cases}
\end{align}
\end{theorem}

\begin{proof}

As $x \to \infty $, from \cite[Formula 5.11.2]{nist}, we have 
\begin{align}\label{psiasym}
	\psi(x) \sim \log(x) -\frac{1}{2 x} - \sum_{m=1}^{\infty} \frac{B_{2m}}{2mx^{2m}}.
\end{align}
where $B_n$ is the \textit{n}$^{\textup{th}}$ Bernoulli number, with the convention that $B_1=\frac{1}{2}$. Let us define the following notation for the line integral:
\begin{align*}
	\int_{(c)} := \int_{c-i \infty}^{c+i \infty}.
\end{align*} 
As given in \cite[Formula 7.3]{ober}, for $0<c<1,$ we have 
\begin{align}
	\psi(x+1)-\log(x)= \frac{1}{2 \pi i} \int_{(c)} \frac{-\pi \zeta(1-s)}{\sin(\pi s)} x^{-s} ds. 
\end{align}
Therefore, by moving the contour to $k+1<c_k<k+2$, we get
\begin{align}\label{contref1}
	\frac{1}{2 \pi i} \int_{(c_k)} \frac{-\pi \zeta(1-s)}{\sin(\pi s)} x^{-s} ds= \psi(x)-\log(x)+\frac{1}{2x}+ \sum_{m=1}^{\lfloor \frac{k+1}{2} \rfloor } \frac{ B_{2m}}{2 m x^{2m}}. 
\end{align}
Since, upon using the special value $\zeta(-n)=-\frac{B_{n+1}}{n+1}$ for any natural number $n$, one can see,
\begin{align}
	\Res_{s=n}  \frac{-\pi \zeta(1-s)}{\sin(\pi s)} x^{-s} = \frac{(-1)^{n} B_n}{ n x^n}.
\end{align}
From \eqref{psiasym}, it is clear that, as $n \to \infty$, 
\begin{align*}
	\psi(nx)-\log(n x)+\frac{1}{2nx}+ \sum_{m=1}^{\lfloor \frac{k+1}{2} \rfloor } \frac{ B_{2m}}{2 m (nx)^{2m}} = O_{x}\left( \frac{1}{n^{k+2}} \right).
\end{align*} 
Therefore replacing $x$ by $nx$ in \eqref{contref1}, multiply the whole expression by $n^k$ and sum over the set of all natural numbers and use it in the definition of $F_{-k,0}(x)$ as follows
\begin{align}\label{contref2}
		F_{-k,0}(x)&= \sum_{
		n=1}^{\infty} n^k \left(\psi(nx)-\log(n x)+\frac{1}{2nx}+ \sum_{m=1}^{\lfloor \frac{k+1}{2} \rfloor } \frac{ B_{2m}}{2 m (nx)^{2m}}\right) \nonumber \\ 
		& = \sum_{n=1}^{\infty} n^k \left(\frac{1}{2 \pi i} \int_{(c_k)} \frac{-\pi \zeta(1-s)}{\sin(\pi s)} (nx)^{-s} ds\right)\nonumber\\
	&=	\frac{1}{2 \pi i} \int_{(c_k)} \frac{-\pi \zeta(1-s)\zeta(s-k)}{\sin(\pi s)} x^{-s} ds,
\end{align}
in the last step we have used the series representation of $\zeta(s-k)$ as $\Re(s)>k+1$.
Let us make the change of variable $w=k+1-s$ in the integral in \eqref{contref2} (and hence $1-s=w-k$ and $1-w=s-k$). Since $k+1<\Re(s)<k+2$, we have $-1<\Re(w)<0$. Hence, we get,
\begin{align*}
		F_{-k,0}(x)	&= \frac{(-1)^k}{x^{k+1}}	\frac{1}{2 \pi i} \int_{(c_{-2})} \frac{-\pi \zeta(w-k)\zeta(1-w)}{\sin(\pi w)} x^{w} dw
\end{align*}
where $-1<c_{-2}<0$. Now, move the line of integration back to $k+1<c_k<k+2$, to get
\begin{align}\label{final}
		F_{-k,0}(x)& =  \frac{(-1)^k}{x^{k+1}} \left(	\frac{1}{2 \pi i} \int_{(c_{k})}
	 \frac{-\pi \zeta(w-k)\zeta(1-w)}{\sin(\pi w)} x^{w} dw - P(x)\right)\nonumber \\
	&=  \frac{(-1)^k}{x^{k+1}} \left(	F_{-k,0}\left(\frac{1}{x}\right) - P(x)\right),
\end{align}
where $P(x)$ is the residue of the integrand at poles, and we have used \eqref{contref2} in the last step. At $1 \leq m \leq k$, since the zeta factors do not have a pole, we have
\begin{align}
\textup{Res}_{w=m} \frac{-\pi \zeta(w-k)\zeta(1-w)}{\sin(\pi w)} x^{w}  = (-1)^{m+1} \zeta(m-k)\zeta(1-m) x^m.
\end{align}
At $w=0$ and $w=k+1$, we encounter a second order pole, with residues
	$ -\gamma \zeta(-k) + \log(x)  \zeta(-k) + \zeta'(-k)$ and $(-1)^k x^{k+1} (\gamma \zeta(-k)+ \log(x) \zeta(-k)- \zeta'(-k))$,
respectively. Therefore, as defined in \eqref{residualterm}, one can write
\begin{align}\label{residue}
P(x)=	\sum_{m=0}^{k+1} P_{m,-k}(x).
\end{align}
Substituting \eqref{residue} in \eqref{final}, and simplify to get
\begin{align}\label{nonsym}
	x^{\frac{k+1}{2}} F_{-k,0}(x) =	(-1)^k\frac{F_{-k,0}\left(\frac{1}{x}\right) }{x^{\frac{k+1}{2}}} 	+ (-1)^{k+1}\frac{1}{x^{\frac{k+1}{2}}} \sum_{m=0}^{k+1} P_{m,-k}(x)  .
\end{align}
Replace $x$ by $\frac{1}{x}$ in \eqref{nonsym} to get
\begin{align}\label{nonsym2}
\frac{F_{-k,0}\left(\frac{1}{x}\right) }{x^{\frac{k+1}{2}}} 	 =	(-1)^k x^{\frac{k+1}{2}} F_{-k,0}(x)	+ (-1)^{k+1} x^{\frac{k+1}{2}} \sum_{m=0}^{k+1} P_{m,-k}\left(\frac{1}{x} \right).
\end{align}
Therefore, from \eqref{nonsym} and \eqref{nonsym2}, one can see that,
\begin{align}\label{halfsum}
(-1)^{k+1}\frac{1}{x^{\frac{k+1}{2}}} \sum_{m=0}^{k+1} P_{m,-k}(x) =x^{\frac{k+1}{2}} \sum_{m=0}^{k+1} P_{m,-k}\left(\frac{1}{x} \right).
\end{align}
Using \eqref{halfsum} in \eqref{nonsym}, we get the required two-term functional equation after simplification.
\end{proof}
As the first special case, we note below the following corollary of the Theorem \ref{newFE}.
\begin{corollary}
	The first equality of Ramanujan's identity \eqref{w1.26} holds.
\end{corollary}
\begin{proof}
	Put $k=0$ in Theorem \ref{newFE} and simplify.
\end{proof}
	
	\section{Equivalence Theorems}\label{Section3}

We start this section by proving a recursive formula for $\Theta(r,s,t,x)$. The motivation for this result originated from a recursive result for $\zeta_{\textup{MT}}(r,s,t)$ obtained by Huard, Williams and Zhang \cite[Equation (1.6)]{huard}. We use the specific case of Theorem \ref{molty} with $r=s=n$, as stated in Corollary \ref{Trrcor}, to prove Theorem \ref{MAIN}. Theorem \ref{molty} can be directly seen as an application of the binomial theorem written in the form
\begin{align*}
	\frac{1}{a^r b^s} = \sum_{\ell=0}^{N} \binom{N}{\ell} \frac{1}{a^{r-(N-\ell)} b^{s-\ell} (a+b)^{N}},
\end{align*}
with the substitutions $a=n$ and $b=mx$. For sake of completion, we give a proof using induction, very briefly. 

		\begin{theorem}\label{molty}
		For any $r,s,t \in \mathbb{R}$ such that $r+s+t>2$, $r+t>1$ and $s+t>1$ and for any $n \in \mathbb{N}$, we have,
		\begin{align}
			\Theta(r,s,t,x) = \sum_{\ell=0}^{n} \binom{n}{\ell} x^{\ell} \Theta(r-n+\ell,s-\ell,t+n,x). \label{molltype1}
		\end{align}
	\end{theorem}
	\begin{proof}
		For $n=1$, the statement holds true, as can be seen from \eqref{Tsplit}.
		We now assume the identity is true for all $j<n$. We start with
		\begin{align*}
			\Theta(r,s,t,x) = \sum_{\ell=0}^{n-1} \binom{n-1}{\ell} x^{\ell} \Theta(r-(n-1)+\ell,s-\ell,t+n-1,x).
		\end{align*}
		Use \eqref{Tsplit} for the summand on the right-hand side of the above equation to get,
		\begin{align*}
			&\Theta(r,s,t,x) = \sum_{\ell=0}^{n-1} \binom{n-1}{\ell} x^{\ell} \left( \Theta(r-n+\ell,s-\ell,t+n,x)+ x \Theta(r-(n-1)+\ell,s-\ell-1,t+n,x) \right)
		\end{align*}
		Replace the variable of summation $\ell$ by $\ell-1$ in the second term of the summand on the right-hand side to get,
		\begin{align*}
			\Theta(r,s,t,x) &= \sum_{\ell=0}^{n-1} \binom{n-1}{\ell} x^{\ell}  \Theta(r-n+\ell,s-\ell,t+n,x)+ \sum_{\ell=1}^{n} \binom{n-1}{\ell-1} x^{\ell} \Theta(r-n+\ell,s-\ell,t+n,x)\\
			&= \sum_{\ell=0}^{n} \binom{n}{\ell}  x^{\ell}  \Theta(r-n+\ell,s-\ell,t+n,x),
		\end{align*}
		where we have used the property $\binom{n-1}{\ell} + \binom{n-1}{\ell-1} = \binom{n}{\ell}$. Observe that $\Theta$ by definition is well defined at the arguments in all of the above steps. Hence, by induction, \eqref{molltype1} holds true for all $n \in \mathbb{N}$.
	\end{proof}
	\begin{corollary}\label{Trrcor}
		For $r,t \in \mathbb{N}$ such that $r+t>1$ and $r>1$, and for any $x>0$,
		\begin{align*}
			\Theta(r,r,t,x) = \sum_{\ell=0}^{r} \binom{r}{\ell} x^{\ell} \Theta(\ell,r-\ell,t+r,x).
		\end{align*}
	\end{corollary}
	\begin{proof}
		Put $s=r$ and $n=r$ in \eqref{molltype1} to get this result.
	\end{proof}
We are now  equipped to prove Theorem \ref{MAIN}.
\begin{proof}[Theorem \textup{\ref{MAIN}}][]\
	We give the proof in two cases. Let $a$ be any non-zero fixed integer and $b$ be any non-negative fixed integer.
	
	\noindent \textbf{Case I}: For $a\geq1$.	We know from Corollary \ref{Trrcor}, for Re$(w)>1-a$, that
	\begin{align*}
		\Theta(a,a,w,x) = \sum_{\ell=0}^{a} \binom{a}{\ell} x^{\ell} \Theta(\ell,a-\ell,w+a,x).
	\end{align*}
	Separate the first and last terms from the sum on the right-hand side to get,
	\begin{align}
		\Theta(a,a,w,x) = \Theta(0,a,w+a,x) + \sum_{\ell=1}^{a-1} \binom{a}{\ell} x^{\ell} \Theta(\ell,a-\ell,w+a,x) + x^{a} \Theta(a,0,w+a,x). \label{N1}
	\end{align}
	Moreover, upon adding and subtracting suitable terms from the summand (as explained in Section \ref{introduction}), we can see that,
	\begin{align}
		\Theta(0,a,w+a,x)\bigg{|}_{w \to -a+b+1} \equiv \pm F_{a,b}(x), \label{N2}
	\end{align}
	possibly,  upto a rational multiple, modulo finitely many special values of $\zeta(s)$. Hence, similarly,
	\begin{align}
		x^{a} \Theta(a,0,w+a,x)\bigg{|}_{w \to -a+b+1} = x^{-w} \Theta \left(0,a,w+a,\frac{1}{x}\right)\bigg{|}_{w \to -a+b+1} \equiv \pm x^{-(-a+b+1)} F_{a,b} \left( \frac{1}{x} \right). \label{N3}
	\end{align}
	Finally, from \eqref{N1}, \eqref{N2} and \eqref{N3}, one can see that,
	\begin{align*}
		\lim_{w\to -a+b+1} \left( \Theta(a,a,w,x) - \sum_{\ell=1}^{a-1} \binom{a}{\ell} x^{\ell} \Theta(\ell,a-\ell,w+a,x)  \right) \equiv \pm  F_{a,b}(x) \pm x^{-(-a+b+1)} F_{a,b} \left( \frac{1}{x} \right),
	\end{align*}
	which is the required equivalence associated with weight $a+b+1$.
	
	\noindent \textbf{Case II}: For $a<0$. We know from Theorem \ref{molty}, for Re$(w)>2$, that
	\begin{align*}
		\Theta(0,0,w,x) = \sum_{\ell=0}^{-a} \binom{-a}{\ell} x^{\ell}  \Theta(a+\ell,-\ell,w-a,x).
	\end{align*}
	Separate the first and last term to get,
	\begin{align}
		\Theta(0,0,w,x) = \Theta(a,0,w-a,x) + \sum_{\ell=1}^{-a-1} \binom{-a}{\ell} x^{\ell}  \Theta(a+\ell,-\ell,w-a,x) + x^{-a} \Theta(0,a,w-a,x). \label{N4}
	\end{align}
	Upon adding and subtracting suitable summands (as explained in Section \ref{introduction}) and letting $w \to a+b+1$, we can get,
	\begin{align}
		\Theta(0,a,w-a,x) \bigg|_{w\to a+b+1} \equiv \pm F_{a,b}(x), \label{N5} 
	\end{align}
	possibly, as explained in the previous case, upto to a rational multiple, modulo finitely many special values of $\zeta(s)$. Hence,
	\begin{align}
		\Theta(a,0,w-a,x) \bigg|_{w\to a+b+1} \equiv \pm x^{-b-1} F_{a,b}\left(\frac{1}{x}\right). \label{N6}
	\end{align}
	Finally, from \eqref{N4}, \eqref{N5} and \eqref{N6}, one can see that,
	\begin{align*}
		\lim_{w\to a+b+1} \left( \Theta(0,0,w,x) - \sum_{\ell=1}^{-a-1} \binom{-a}{\ell} x^{\ell}  \Theta(a+\ell,-\ell,w-a,x) \right) =  \pm x^{-b-1} F_{a,b}\left(\frac{1}{x}\right) \pm x^{-a} F_{a,b}(x),
	\end{align*}
	which is the required equivalence associated with weight $a+b+1$.
\end{proof}

Our next aim is to make the five examples listed in Section \ref{introduction} explicit. We break its details into smaller propositions for ease of reading. As mentioned earlier, it has been shown in \cite[Corollary 3.5]{dss2} that \eqref{principal part equation} implies the two-term functional equation \eqref{fe2} of $F(x)$. In the below proposition we prove its converse implication, which together complete Example \ref{ex1}.

\begin{proposition}\label{ZtoKLF}
	For $x>0$, let $F(x) $ and $\Theta(1,1,t,x)$ be as defined in \eqref{herglotzdef} and \eqref{Tdef}, respectively. The two-term functional equation of $F(x)$ given in \eqref{fe2} implies the Kronecker limit formula \eqref{principal part equation} for $\Theta(1,1,t,x)$ around $t=0$.
\end{proposition}

\begin{proof}
	Rearrange the terms in \eqref{decomposition} as follows
	\begin{align*}
		\Theta(1,1,z-1,x) - \frac{(1+x^{1-z})}{z-1}\zeta(z) = \Phi(z, x)+x^{1-z}\Phi\left(z, \frac{1}{x}\right) - (x+x^{-z})\zeta(z+1).
	\end{align*}
	Now tend limit $z\to 1$ to get,
	\begin{align*}
		\lim_{z\to 1} \left( \Theta(1,1,z-1,x) - \frac{(1+x^{1-z})}{z-1}\zeta(z) \right) = -F(x) - F\left(\frac{1}{x}\right) - \left(x + \frac{1}{x}\right) \zeta(2),
	\end{align*}
	where we have used $\lim_{z\to 1} \big(\zeta(z,y)-\frac{1}{z-1}\big) = -\psi(y)$ and L'Hospital's rule to obtain the limit on the right-hand side. Use \eqref{fe2} to get,
	\begin{align*}
		\lim_{z\to 1} \left( \Theta(1,1,z-1,x) - \frac{(1+x^{1-z})}{z-1}\zeta(z) \right) = - \left( 2F(1)+\frac{1}{2}\log^{2}(x)-\frac{\pi^2}{6x}(x-1)^2 \right) - \left(x + \frac{1}{x}\right) \zeta(2).
	\end{align*}
	Hence, using \eqref{ef1}, in a small neighbourhood of $z=1$, we have
	\begin{align*}
		\Theta(1,1,z-1,x) - \frac{(1+x^{1-z})}{z-1}\zeta(z) = \gamma^2 -\frac{\pi^2}{6} -\frac{\log^2(x)}{2} +\gamma_1 +  O_{x}(|z-1|).
	\end{align*} 
	Using the Laurent series expansions of $x^{1-z}$ and $\zeta(z)$ around $z=1$, and replacing $z-1=t$ we get \eqref{principal part equation}.
\end{proof} 
 The following two Propositions \ref{KLFtoVZ} and \ref{VZtoKLF} together elaborate Example \ref{ex2}.

	\begin{proposition}\label{KLFtoVZ}
	Let $x>0$ and $r \in \mathbb{N}, r>1$. Let $F_r(x) $, and $\Theta(r,r,t,x)$ be as defined in \eqref{higherherglotz}, and \eqref{Tdef} respectively. The Kronecker limit formula \eqref{KLF-r}  for $\Theta(r,r,t,x)$ around $t=1-r$ implies the two-term functional equation of $F_r(x)$ given in \eqref{vz2term}.
	\end{proposition}
			\begin{proof}
			Let $r \in \mathbb{N}$, $r\ne1$ be arbitrary and $t>1-r$. Let $x>0$. From Corollary \ref{Trrcor}, we have,
			\begin{align*}
				\Theta(r,r,t,x) = \sum_{\ell=0}^{r} \binom{r}{\ell} x^{\ell} \Theta(\ell,r-\ell,t+r,x).
			\end{align*}
			Separate out the first and the last term from the sum to get,
			\begin{align*}
				\Theta(r,r,t,x) =\Theta(0,r,t+r,x) + \sum_{\ell=1}^{r-1} \binom{r}{\ell} x^{\ell} \Theta(\ell,r-\ell,t+r,x) + x^r \Theta(r,0,t+r,x).
			\end{align*}
			Use inversion \eqref{Tinv} for the third term on the right-hand side to get,
			\begin{align}
				\Theta(r,r,t,x) =\Theta(0,r,t+r,x) + S(r,t) + x^{-t} \Theta(0,r,t+r,\tfrac{1}{x}), \label{rlevel}
			\end{align}
			where $S(r,t)$ is defined by
			\begin{align*}
				S(r,t):= \sum_{\ell=1}^{r-1} \binom{r}{\ell} x^{\ell} \Theta(\ell,r-\ell,t+r,x).
			\end{align*}
			We now simplify $S(r,t)$ as follows:
			\begin{align*}
				S(r,t)&= \sum_{\ell=1}^{r-1} \left( \sum_{i=1}^{\ell} (-1)^{\ell-i} \binom{r}{i} -  \sum_{j=1}^{\ell-1} (-1)^{\ell-j} \binom{r}{j} \right) x^{\ell} \Theta(\ell,r-\ell,t+r,x) \\
				&= \sum_{\ell=1}^{r-1} \sum_{i=1}^{\ell} (-1)^{\ell-i} \binom{r}{i}  x^{\ell} \Theta(\ell,r-\ell,t+r,x) - \sum_{\ell=1}^{r-1} \sum_{j=1}^{\ell-1} (-1)^{\ell-j} \binom{r}{j} x^{\ell} \Theta(\ell,r-\ell,t+r,x).
			\end{align*}
			Replace $\ell$ by $\ell+1$ in the second term of the right-hand side to get,
			\begin{align*}
				S(r,t)&= \sum_{\ell=1}^{r-1} \sum_{i=1}^{\ell} (-1)^{\ell-i} \binom{r}{i}  x^{\ell} \Theta(\ell,r-\ell,t+r,x) + \sum_{\ell=0}^{r-2} \sum_{j=1}^{\ell} (-1)^{\ell-j} \binom{r}{j} x^{\ell+1} \Theta(\ell+1,r-\ell-1,t+r,x).
			\end{align*}
			Since $\ell=0$ makes the summand zero, hence the second term can be rewritten as,
			\begin{align*}
				S(r,t)&= \sum_{\ell=1}^{r-1} \sum_{i=1}^{\ell} (-1)^{\ell-i} \binom{r}{i}  x^{\ell} \Theta(\ell,r-\ell,t+r,x) + \sum_{\ell=1}^{r-2} \sum_{j=1}^{\ell} (-1)^{\ell-j} \binom{k}{j} x^{\ell+1} \Theta(\ell+1,r-\ell-1,t+r,x).
			\end{align*}
			Separate the $\ell=r-1$ term from the first term on the right-hand side to get,
			\begin{align}
				S(r,t)&= x^{r-1} \Theta(r-1,1,t+r,x) \sum_{i=1}^{r-1}(-1)^{r-i+1}  \binom{r}{i} \notag \\
				&\quad+\sum_{\ell=1}^{r-2} \sum_{i=1}^{\ell} (-1)^{\ell-i} \binom{r}{i}  x^{\ell} \Theta(\ell,r-\ell,t+r,x) + \sum_{\ell=1}^{r-2} \sum_{j=1}^{\ell} (-1)^{\ell-j} \binom{k}{j} x^{\ell+1} \Theta(\ell+1,r-\ell-1,t+r,x)\notag\\
				&= x^{r-1} \Theta(r-1,1,t+r,x) \sum_{i=1}^{r-1}(-1)^{r-i+1}  \binom{r}{i}\notag\\
				&\quad +\sum_{\ell=1}^{r-2} \sum_{i=1}^{\ell} (-1)^{\ell-i} \binom{r}{i}  x^{\ell} \bigg( \Theta(\ell,r-\ell,t+r,x) +  x \Theta(\ell+1,r-\ell-1,t+r,x) \bigg)\notag\\
				&= x^{r-1} \Theta(r-1,1,t+r,x) \sum_{i=1}^{r-1}(-1)^{r-i+1}  \binom{r}{i}+\sum_{\ell=1}^{r-2} \sum_{i=1}^{\ell} (-1)^{\ell-i} \binom{r}{i}  x^{\ell}  \Theta(\ell+1,r-\ell,t+r-1,x), \label{srt simp}
			\end{align}
			where we have used \eqref{Tsplit} in the last step. We note the following fact obtained by the Binomial theorem,
			\begin{align*}
				\sum_{i=1}^{r-1}(-1)^{r-i+1}\binom{r}{i}  \buildrel \rm \emph{i} \rightarrow \emph{r}-\emph{i} \over =  -\sum_{i=1}^{r-1}(-1)^{i}\binom{r}{i} =1+(-1)^r = \left\{
				\begin{array}{ll}
					0, & \text{if r is odd}, \\
					2, & \text{if r is even}. \\
				\end{array} 
				\right.
			\end{align*}
			Also, we know that 
			\begin{align*}
				\sum_{i=1}^{\ell}  (-1)^{\ell-i} \binom{r}{i}  = 	(-1)^{\ell} \sum_{i=1}^{\ell}  (-1)^{i}  \binom{r}{i} = (-1)^{\ell} \left(-1+	\sum_{i=0}^{\ell}  (-1)^{i}  \binom{r}{i} \right) = (-1)^{\ell+1} + \binom{r-1}{\ell},
			\end{align*}
			where we have used that $\sum_{i=0}^{\ell}  (-1)^{i}  \binom{r}{i} = (-1)^{\ell}\binom{r-1}{\ell}$.  Hence, using the two facts mentioned above, \eqref{srt simp} becomes,
			\begin{align}
				S(r,t)= (1+(-1)^r)x^{r-1} \Theta(r-1,1,t+r,x) +\sum_{\ell=1}^{r-2} \left( (-1)^{\ell+1} + \binom{r-1}{\ell} \right)  x^{\ell}  \Theta(\ell+1,r-\ell,t+r-1,x). \label{srt0eqn}
			\end{align}
			Take the limit $t \to 1-r$ in \eqref{srt0eqn} to get,
			\begin{align}
				\lim_{t\to 1-r} S(r,t)= (1+(-1)^r)x^{r-1} \lim_{t\to 1-r}  \Theta(r-1,1,t+r,x) +\sum_{\ell=1}^{r-2} \left( (-1)^{\ell+1} + \binom{r-1}{\ell} \right)  x^{\ell} \zeta(\ell+1) \zeta(r-\ell), \label{srt1eqn}
			\end{align}
			since $\Theta(a,b,0,x)=\zeta(a)\zeta(b)$ for any $a,b>1$. The limit on right-hand side is dealt as follows,
			\begin{align}
				\lim_{t\to 1-r}  \Theta(r-1,1,t+r,x) &= \lim_{t \to 1-r}  x^{-t-r}\Theta(1,r-1,t+r,\tfrac{1}{x}) \notag \\ 
				&= \frac{1}{x} \sum_{m=1}^{\infty} \frac{1}{m^{r-1}} \sum_{n=1}^{\infty}  \frac{1}{n (n+\frac{m}{x})} \notag\\
				&= \frac{1}{x} \sum_{m=1}^{\infty} \frac{1}{m^{r-1}} \sum_{n=1}^{\infty} \left( \frac{\frac{x}{m}}{n} - \frac{\frac{x}{m}}{n+\frac{m}{x}} \right)  \notag\\
				&=  \sum_{m=1}^{\infty} \frac{1}{m^{r}} \left( \gamma + \psi(\tfrac{m}{x}+1) \right)  \notag\\
				&= \gamma \zeta(r) + \sum_{m=1}^{\infty} \frac{\psi(\frac{m}{x}+1)}{m^{r}} \label{Tlim}
			\end{align}
			Hence, from \eqref{srt1eqn} and \eqref{Tlim}, we get,
			\begin{align}
				\lim_{t\to 1-r} S(r,t) &= (1+(-1)^r)x^{r-1} \left( \sum_{m=1}^{\infty} \frac{\psi(\frac{m}{x}+1)}{m^{r}} +\gamma \zeta(r) \right) +\sum_{\ell=1}^{r-2} \left( (-1)^{\ell+1} + \binom{r-1}{\ell} \right)  x^{\ell} \zeta(\ell+1) \zeta(r-\ell). \label{srtLim}
			\end{align}
			We also need the following evaluation,
			\begin{align}
				\Theta(0,r,t+r,x)&= \sum_{n=1}^{\infty} \sum_{m=1}^{\infty} \frac{1}{m^r (n+mx)^{t+r}} \notag \\
				&= \sum_{m=1}^{\infty} \frac{\zeta(t+r,mx+1)}{m^r} \notag \\
				&= \sum_{m=1}^{\infty} \frac{\zeta(t+r,mx+1) -\frac{1}{t-1+r}}{m^r} + \frac{\zeta(r)}{t-1+r}. \label{Teval}
			\end{align}
			Hence, we can also see that,
			\begin{align}
				x^{-t} \Theta(0,r,t+r,\tfrac{1}{x}) = x^{-t} \left( \sum_{m=1}^{\infty} \frac{\zeta(t+r,\tfrac{m}{x}+1) -\frac{1}{t-1+r}}{m^r} + \frac{\zeta(r)}{t-1+r} \right). \label{Teval1}
			\end{align} 
			Substitute \eqref{Teval} and \eqref{Teval1} in \eqref{rlevel} to get,
			\begin{align*}
				&\Theta(r,r,t,x) - \frac{\zeta(r) (1+x^{-t})}{t-1+r} \\
				&= \sum_{m=1}^{\infty} \frac{\zeta(t+r,mx+1) -\frac{1}{t-1+r}}{m^r}  + S(r,t) + x^{-t}  \sum_{m=1}^{\infty} \frac{\zeta(t+r,\tfrac{m}{x}+1) -\frac{1}{t-1+r}}{m^r} .
			\end{align*}
			Take the limit of both sides of the above equation as $t \to 1-r$, and use \eqref{srtLim} to get
			\begin{align}
				&\lim_{t \to 1-r} \left( \Theta(r,r,t,x) - \frac{\zeta(r) (1+x^{-t})}{t-1+r} \right) \notag \\ 
				&= -\sum_{m=1}^{\infty} \frac{\psi(mx+1)}{m^r} + \lim_{t \to 1-r} (S(r,t)) - x^{r-1} \sum_{m=1}^{\infty} \frac{\psi(\frac{m}{x}+1)}{m^r}\notag \\
				&= -\sum_{m=1}^{\infty} \frac{\psi(mx+1)}{m^r} + (-1)^r x^{r-1} \sum_{m=1}^{\infty} \frac{\psi(\frac{m}{x}+1)}{m^{r}} +(1+(-1)^r)x^{r-1} \gamma \zeta(r) \notag\\
				&\quad  +\sum_{\ell=1}^{r-2} \left((-1)^{\ell+1} + \binom{r-1}{\ell} \right)  x^{\ell} \zeta(\ell+1) \zeta(r-\ell). \label{aftLim}
			\end{align}
		From \eqref{KLF-r}, we have
			\begin{align}
				\Theta(r,r,t,x) &= \frac{\zeta(r)\left(1 + x^{r-1} \right)}{(t-1+r)}  +  x^{r-1} \zeta(r)(\gamma-\log(x)) \notag \\
				& \quad + \gamma \zeta(r)  + x^{r-1} \sum_{\substack{k=1}}^{r-2} \binom{r-1}{k}  \frac{1}{ x^{k}} \zeta(r-k)\zeta(k+1)   + O_{x}(|t-1+r|).  \notag
			\end{align}
			Upon adding and subtracting suitable terms, we get,
			\begin{align}
				&\lim_{t \to 1-r} \left( \Theta(r,r,t,x) - \frac{\zeta(r) (1+x^{-t})}{t-1+r} \right) \notag \\
				&= \lim_{t \to 1-r} \left( \frac{\zeta(r)\left(x^{r-1}-x^{-t}\right)}{t-1+r} \right)  +  x^{r-1} \zeta(r)(\gamma-\log(x))  + \gamma \zeta(r)  + x^{r-1} \sum_{\substack{k=1}}^{r-2} \binom{r-1}{k}  \frac{1}{ x^{k}} \zeta(r-k)\zeta(k+1) \notag \\
				&= x^{r-1} \zeta(r) \log(x) +  x^{r-1} \zeta(r)(\gamma-\log(x)) + \gamma \zeta(r)  + x^{r-1} \sum_{\substack{k=1}}^{r-2} \binom{r-1}{k}  \frac{1}{ x^{k}} \zeta(r-k)\zeta(k+1) \notag \\
				&= x^{r-1} \zeta(r)\gamma + \zeta(r)  \gamma  +  \sum_{\substack{k=1}}^{r-2} \binom{r-1}{k}  x^{k} \zeta(r-k)\zeta(k+1), \label{Tconst}
			\end{align}
			where we have replaced $k$ by $r-1-k$ in the last step.
			Hence, from \eqref{aftLim} and \eqref{Tconst}, we see,
			\begin{align*}
				&-\sum_{m=1}^{\infty} \frac{\psi(mx+1)}{m^r} + (-1)^r x^{r-1} \sum_{m=1}^{\infty} \frac{\psi(\frac{m}{x}+1)}{m^{r}} +(1+(-1)^r)x^{r-1} \gamma \zeta(r) \notag\\
				&+\sum_{\ell=1}^{r-2} \left( (-1)^{\ell+1} + \binom{r-1}{\ell} \right)  x^{\ell} \zeta(\ell+1) \zeta(r-\ell) \\
				&= x^{r-1} \zeta(r)\gamma + \zeta(r)  \gamma  +  \sum_{\substack{k=1}}^{r-2} \binom{r-1}{k}  x^{k} \zeta(r-k)\zeta(k+1).
			\end{align*}
			On simplification, we get, with the convention that $\zeta(1)=\gamma$ as taken in \cite{vz}, we get
			\begin{align*}
				\sum_{m=1}^{\infty} \frac{\psi(mx+1)}{m^r} + (-x)^{r-1} \sum_{m=1}^{\infty} \frac{\psi(\frac{m}{x}+1)}{m^{r}} &= (-1)^{r} x^{r-1} \gamma \zeta(r) -\sum_{\ell=1}^{r-2} (-1)^{\ell}  x^{\ell} \zeta(\ell+1) \zeta(r-\ell)   -\zeta(r)  \gamma. \\
				&= -\sum_{\ell=0}^{r-1} (-x)^{\ell} \zeta(\ell+1) \zeta(r-\ell).
			\end{align*}
			Use the functional equation $\psi(y+1)=\psi(y)+ y^{-1}$ in both the summands on the left-hand side of the above equation and simplify to get \eqref{vz2term}.
	\end{proof}
	
	\begin{proposition}\label{VZtoKLF}
	Let $x>0$ and $r \in \mathbb{N}$, $r>1$. Let $F_r(x)$ and $\Theta(r,r,t,x)$ be as defined in \eqref{higherherglotz}, and \eqref{Tdef} respectively. The two-term functional equation of $F_r(x)$ given in \eqref{vz2term} implies Kronecker limit formula \eqref{KLF-r}  for $\Theta(r,r,t,x)$ around $t=1-r$.
	\end{proposition}
\begin{proof}
	Form \eqref{rlevel}, we have,
	\begin{align}
		\Theta(r,r,t,x) = \Theta(0,r,t+r,x) + S(r,t) + x^{-t} \Theta(0,r,t+r,\tfrac{1}{x}), \label{start11}
	\end{align}
	where
	\begin{align*}
		S(r,t)= \sum_{\ell=1}^{r-1} \binom{r}{\ell} x^{\ell} \Theta(\ell,r-\ell,t+r,x).
	\end{align*}
	From \eqref{srtLim}, we have,
	\begin{align*}
		\lim_{t\to 1-r} S(r,t) &= (1+(-1)^r)x^{r-1} \left( \sum_{m=1}^{\infty} \frac{\psi(\frac{m}{x}+1)}{m^{r}}  +\gamma \zeta(r) \right) +\sum_{\ell=1}^{r-2} \left( (-1)^{\ell+1} + \binom{r-1}{\ell} \right)  x^{\ell} \zeta(\ell+1) \zeta(r-\ell).
	\end{align*}
	Since $S(r,t)$ is analytic in $t$ in a small neighbourhood of $t=1-r$, we can rewrite is as,
	\begin{align}
		S(r,t) &= (1+(-1)^r)x^{r-1} \left( \sum_{m=1}^{\infty} \frac{\psi(\frac{m}{x}+1)}{m^{r}} +\gamma \zeta(r) \right) \notag \\
		&\quad +\sum_{\ell=1}^{r-2} \left( (-1)^{\ell+1} + \binom{r-1}{\ell} \right)  x^{\ell} \zeta(\ell+1) \zeta(r-\ell) +O_{x}(|t-(1-r)|). \label{srtLim11}
	\end{align}
	From \eqref{rlevel}, we rearrange the terms and tend $t\to 1-r$ to get,
	\begin{align*}
		\lim_{t\to 1-r}\left(\Theta(0,r,t+r,x)-\frac{\zeta(r)}{t-(1-r)}\right) = - \sum_{m=1}^{\infty} \frac{\psi(mx+1)}{m^r}.
	\end{align*}
	In a small neighbourhood of $t=1-r$, we can rewrite is as,
	\begin{align}
		\Theta(0,r,t+r,x) = \frac{\zeta(r)}{t-(1-r)} - \sum_{m=1}^{\infty} \frac{\psi(mx+1)}{m^r} + O_{x}(|t-(1-r)|). \label{TTlim1}
	\end{align}
	Note, in a neighborhood of $t=1-r$,
	\begin{align}
		x^{-t} = x^{r-1} - x^{r-1} \log(x) (t-(1-r)) +  O_{x}(|t-(1-r)|^2). \label{xxlim1}
	\end{align}
	Replace $x$ by $\frac{1}{x}$ in \eqref{TTlim1} and use \eqref{xxlim1} to get, in a small neighbourhood of $t=1-r$, 
	\begin{align}
		x^{-t} \Theta(0,r,t+r,\tfrac{1}{x}) &= \bigg( x^{r-1}-x^{r-1}\log(x) (t-(1-r)) +  O_{x}(|t-(1-r)|^2) \bigg) \notag
		\\&\quad \times \left( \frac{\zeta(r)}{t-(1-r)} - \sum_{m=1}^{\infty} \frac{\psi(\frac{m}{x}+1)}{m^r} + O_{x}(|t-(1-r)|) \right) \notag \\
		&= \frac{x^{r-1}\zeta(r)}{t-(1-r)} - x^{r-1} \sum_{m=1}^{\infty} \frac{\psi(\frac{m}{x}+1)}{m^r} -x^{r-1} \log(x) \zeta(r)  +O_{x}(|t-(1-r)|). \label{TTlim2}
	\end{align}
	Using \eqref{srtLim11}, \eqref{TTlim1} and \eqref{TTlim2} in \eqref{start11}, we can see that, in a small neighbourhood of $t=1-r$, 
	\begin{align*}
		\Theta(r,r,t,x) &= (1+x^{r-1})\frac{\zeta(r)}{t-(1-r)} - \sum_{m=1}^{\infty} \frac{\psi(mx+1)}{m^r} -(-x)^{r-1} \sum_{m=1}^{\infty} \frac{\psi(\frac{m}{x}+1)}{m^{r}} + (1+(-1)^r)x^{r-1} \left( \gamma \zeta(r) \right) \\
		&\quad -x^{r-1} \log(x) \zeta(r)  +\sum_{\ell=1}^{r-2} \left( (-1)^{\ell+1} + \binom{r-1}{\ell} \right)  x^{\ell} \zeta(\ell+1) \zeta(r-\ell) +O_{x}(|t-(1-r)|).
	\end{align*}
	Now, use \eqref{vz2term} to get,
	\begin{align*}
		\Theta(r,r,t,x) &= (1+x^{r-1})\frac{\zeta(r)}{t-(1-r)} - \zeta(r+1) \left( (-x)^{r} -\frac{1}{x} \right) \notag \\
		&\quad + \sum_{\ell=1}^{r} \zeta(\ell) \zeta(r-\ell+1) (-x)^{\ell-1} + (1+(-1)^r)x^{r-1} \left( \gamma \zeta(r) \right) \\
		&\quad -x^{r-1} \log(x) \zeta(r)  +\sum_{\ell=1}^{r-2} \left( (-1)^{\ell+1} + \binom{r-1}{\ell} \right)  x^{\ell} \zeta(\ell+1) \zeta(r-\ell) +O_{x}(|t-(1-r)|).
	\end{align*}
	Simplify and use the convention $\zeta(1)=\gamma$ (as taken in \cite{vz}) to get \eqref{KLF-r}
\end{proof}
The following proposition illustrates Example \ref{ex3}.	

	\begin{proposition}
	The two-term functional equation satisfied by $F_{-1,0}(x)$ is equivalent to the Kronecker limit type formula for $\Theta(0,0,t,x)$ around $t = 0$. 
\end{proposition}
\begin{proof}
	
	Upon using each \eqref{Tsplit} and \eqref{Tinv} once, we can get,
	\begin{align}
		\Theta(0,0,w-1,x)=\Theta(-1,0,w,x)+ x^{1-w} \Theta(-1,0,w,\tfrac{1}{x}). \label{ESS3A}
	\end{align}	
	Firstly, we simplify $\Theta(-1,0,w,\tfrac{1}{x})$ as follows:
	\begin{align}
		\Theta(-1,0,w,\tfrac{1}{x}) &= \sum_{n=1}^{\infty} \sum_{m=1}^{\infty} \frac{n}{(n+ \tfrac{m}{x})^{w}} \notag\\
		&= x^{w} \sum_{n=1}^{\infty} n \zeta(w,nx) - \zeta(w-1), \notag \\
		&= x^{w} \sum_{n=1}^{\infty} n \left( \zeta(w,nx) - \frac{(nx)^{1-w}}{w-1} - \frac{(nx)^{-w}}{2} - \frac{1}{12}w(nx)^{-1-w} \right) \notag \\
		&\quad  + \frac{x}{w-1} \zeta(w-2) + \frac{1}{2} \zeta({w-1}) + \frac{w}{12x} \zeta({w}) -\zeta(w-1). \label{ESS4}
	\end{align}
	Rearrange a few terms in the above equation to get
	\begin{align}
		&\lim_{w \to 1} \left( \Theta(-1,0,w,\tfrac{1}{x}) -  \frac{x}{w-1} \zeta(w-2) - \frac{w}{12x} \zeta({w}) \right) \notag \\ 
		&=  -x \sum_{n=1}^{\infty} n \left( \psi(nx) -\log(nx) + \frac{1}{2nx} + \frac{1}{12n^2x^2} \right)+ \frac{1}{4}. \label{ESS5}
	\end{align}
	Using well-known Laurent series expansions, we get, around $w=1$,
	\begin{align}
		\Theta(-1,0,w,\tfrac{1}{x}) &= \frac{1-x^2}{12 x (w-1)} - x\sum_{n=1}^{\infty} n \left( \psi(nx) -\log(nx) + \frac{1}{2nx} + \frac{1}{12n^2x^2} \right) \notag \\
		& \quad +\frac{1}{12x} \left( 1+\gamma + 3x +x^2-12x^2 \log(A) \right) + O_x(|w-1|), \label{ESS6}
	\end{align}
	where $A$ is the \textit{Glaisher's constant}, and $\log(A) = \frac{1}{12} - \zeta'(-1)$.	Also noting that, around $w=1$, $x^{1-w} = 1 - \log(x) (w-1) + O(|w-1|^2)$. Hence, by substituting \eqref{ESS6} twice in \eqref{ESS3A}, we get, around $w=1$,
	\begin{align}\label{eqvl}
		&\Theta(0,0,w-1,x) \notag \\
		&= - \frac{1}{x} \sum_{n=1}^{\infty} n \left( \psi\left(\frac{n}{x}\right) -\log\left(\frac{n}{x}\right) + \frac{x}{2n} + \frac{x^2}{12n^2} \right) - x \sum_{n=1}^{\infty} n \left( \psi(nx) -\log(nx) + \frac{1}{2nx} + \frac{1}{12n^2x^2} \right) \notag \\
		&\quad  + \frac{1}{12x} \left( 2+\gamma+6x+2x^2+\gamma x^2-12 \log(A)-12x^2 \log(A)-\log(x)+x^2\log(x) \right) + O_x(|w-1|).
	\end{align}	
Hence, on replacing $w-1$ by $t$, one see that the required Kronecker limit type formula can be obtained from the two-term functional equation for $F_{-1,0}(x)$. 
\end{proof}
The above equivalence gives an alternative way to obtain the two-term functional equation for $F_{-1,0}(x)$ using the Kronecker limit type formula for $\Theta(0,0,t,x)$ around $t=0$ obtained the authors previously \cite{ss2}. We note this observation below as a corollary.
\begin{corollary}\label{newFEc}
	The function $F_{-1,0}(x)$ satisfies the following two-term functional equation:
	\begin{align}\label{newfe}
		F_{-1,0}(x)+\frac{1}{x^2}F_{-1,0}\left(\frac{1}{x}\right)
		&= \frac{1}{12 x^2} \left(3x +(1+\gamma-12 \log (A)) (1 + x^2) - (1- x^2) \log(x)\right),
	\end{align}
	where $A$ is the Glaisher's constant.
\end{corollary}
\begin{proof}
	Around $t=0$, from \cite[Equation (1.14)]{ss2}, we know,
	\begin{align}\label{from SS2}
		\Theta(0,0,t,x) = \frac{x}{12} + \frac{1}{4} + \frac{1}{12x} + O_{x}(|t|).
	\end{align}
	By substituting \eqref{from SS2} in \eqref{eqvl} and then tending $t \to 0$, we get \eqref{newfe}.
\end{proof}
We next explain Example \ref{ex4} through the following propostion. 
\begin{proposition}
	For $x>0$, we have,
	\begin{align}
		\sum_{n=1}^{\infty} n \psi^{(3)} \left(1+nx\right) + \frac{1}{x^5} \sum_{n=1}^{\infty} n \psi^{(3)} \left(1+\frac{n}{x}\right) = \frac{6}{x} \Theta(0,0,3,x). \label{example4st}
	\end{align}
\end{proposition}
\begin{proof}
From the split property \eqref{Tsplit}, we have,
\begin{align*}
	\Theta(0,0,3,x) = \Theta(-1,0,4,x) + x \Theta(0,-1,4,x).
\end{align*}
Similar to the calculations as done above, one can see that,
\begin{align*}
	\Theta(-1,0,4,x) = \sum_{n=1}^{\infty} \sum_{m=1}^{\infty} \frac{n}{(n+mx)^4} = \frac{1}{6x^4} \sum_{n=1}^{\infty} n \psi^{(3)} \left(1+\frac{n}{x}\right).
\end{align*}
Upon using $\Theta(0,-1,4,x) = x^{-4} \Theta(-1,0,4,\tfrac{1}{x})$ and simplifying, we get the required identity \eqref{example4st}
\end{proof}
Finally, we explain Example \eqref{ex5} through the following proposition.
\begin{proposition}
	The two-term functional equation for $F_{-3,0}(x)$ is equivalent to the Kronecker limit type formula satisfied by $\Theta(-2,-1,t,x)+ x^{1-t} \Theta\left(-2,-1, t,\frac{1}{x}\right)$ around $t=1$.
\end{proposition}
\begin{proof}
	Again, use \eqref{Tsplit} and \eqref{Tinv} to see that
	\begin{align}\label{useinv}
		\Theta(0,0,t-3,x)=\Theta(-3,0,t,x)+ 3 x \Theta(-2,-1,t,x)+3 x^{2-t} \Theta\left(-2,-1, t,\frac{1}{x}\right)+  x^{3-t}\Theta\left(-3, 0, t,\frac{1}{x}\right)
	\end{align}
	Note that,
	\begin{align}\label{seriesexp}
		\Theta\left(-3,0,t,\frac{1}{x}\right) 
		&= x^t \sum_{n=1}^{\infty} n^3 \zeta\left(t,nx\right) - \zeta(t-3) \nonumber \\
		&=x^t\sum_{n=1}^{\infty} n^3 \left( \zeta(t,nx) - \frac{(nx)^{1-t}}{t-1} - \frac{(nx)^{-t}}{2} - \frac{t (nx)^{-1-t} }{12}+\frac{(t+2)(t+1)t(nx)^{-3-t}}{720}\right) \nonumber \\
		&\quad+ \frac{x}{t-1} \zeta(t-4) - \frac{1}{2} \zeta({t-3}) + \frac{t}{12 x} \zeta({t-2})-\frac{(t+2)(t+1)t}{720 x^3} \zeta(t).
	\end{align}
	Substituting \eqref{seriesexp} in \eqref{useinv} and simplifying, we get
	\begin{align}\label{klfwithz}
		&3\left( x \Theta(-2,-1,t,x)+ x^{2-t} \Theta\left(-2,-1, t,\frac{1}{x}\right)\right) \nonumber\\&=-x^3\sum_{n=1}^{\infty} n^3 \left( \zeta(t,nx) - \frac{(nx)^{1-t}}{t-1} - \frac{(nx)^{-t}}{2} - \frac{t (nx)^{-1-t} }{12}+\frac{(t+2)(t+1)t(nx)^{-3-t}}{720}\right) \nonumber\\	
		&\quad -x^{-t}\sum_{n=1}^{\infty} n^3 \left( \zeta\left(t,\frac{n}{x}\right) - \frac{\left(\frac{n}{x}\right)^{1-t}}{t-1} - \frac{\left(\frac{n}{x}\right)^{-t}}{2} - \frac{t \left(\frac{n}{x}\right)^{-1-t} }{12}+\frac{(t+2)(t+1)t\left(\frac{n}{x}\right)^{-3-t}}{720}\right) \nonumber \\
		&\quad +\frac{(x^{-t}+x^3)(t+2)(t+1)t}{720} \zeta(t) - \frac{(x^{2-t}+x)t}{12} \zeta({t-2}) + \frac{(x^{3-t}+1)}{2}\zeta(t-3) \nonumber \\
		&\quad - \frac{\left(x^{4-t}+\frac{1}{x}\right)}{t-1} \zeta(t-4)  +\Theta(0,0,t-3,x).
	\end{align}
	From the definition of $F_{-k,0}(x)$ for $k=3$ as given in \eqref{lowerherglotz}, clearly, around $t=1$
	\begin{align}\label{lhminus3}
		&\sum_{n=1}^{\infty} n^3 \left( \zeta(t,nx) - \frac{(nx)^{1-t}}{t-1} - \frac{(nx)^{-t}}{2} - \frac{t (nx)^{-1-t} }{12}+\frac{(t+2)(t+1)t(nx)^{-3-t}}{720}\right) = -F_{-3,0}(x)+O_{x}(|t-1|).
	\end{align}
	Choose $r=0$, $s=0$, $t=t-3$, and, $\ell=3,$ in \cite[Theorem 2.3]{ss2} to get 
	\begin{align}\label{ss2klf}
		\Theta(0,0,t-3,x)=-\frac{x^3}{3}\zeta(-3)-\frac{1}{3x}\zeta(-3)+ 2 x \zeta^2(-1)+ O_{x}(|t-1|).
	\end{align}
	Using the well-known Laurent series expansions, we get,
	\begin{align}\label{zetaexp}
		&	\frac{(x^{-t}+x^3)(t+2)(t+1)t}{720} \zeta(t) - \frac{(x^{2-t}+x)t}{12} \zeta({t-2}) + \frac{(x^{3-t}+1)}{2}\zeta(t-3) - \frac{\left(x^{4-t}+\frac{1}{x}\right)}{t-1} \zeta(t-4)\nonumber \\
		&=\frac{1}{720x} \left( 10x^2+ (1+x^4) \left( 11+6\gamma - 720\zeta'(-3) \right) -6(1-x^4) \log(x) \right) + O_{x}(|t-1|).
	\end{align}
	From \eqref{klfwithz}, \eqref{lhminus3}, \eqref{ss2klf}, and \eqref{zetaexp}, we finally get
	\begin{align}
		&\Theta(-2,-1,t,x)+ x^{1-t} \Theta\left(-2,-1, t,\frac{1}{x}\right) = \frac{x^2}{3} \left(F_{-3,0}\left(x\right)+\frac{1}{x^4} F_{-3,0}\left(\frac{1}{x}\right) \right) \nonumber \\
		&\quad + \frac{1}{2160 x^2} \left( 20 x^2 +  (1 + x^4) (9+ 6\gamma - 720\zeta'(-3)) - 6 (1 - x^4) \log(x) \right) +	O_{x}(|t-1|),	\label{klfk3}
	\end{align}
	around $t=1$, which establishes the required equivalence.
\end{proof}
Since we have already derived the functional equation for $F_{-3,0}(x)$ in Theorem \ref{newFE}, we can obtain the Kronecker limit type formula for $\Theta(-2,-1,t,x)+ x^{1-t} \Theta\left(-2,-1, t,\frac{1}{x}\right)$ from the above proposition. In particular, taking $x=1$, we obtain the following special value of $\zeta_{\text{MT}} (r,s,t)$ at the \textit{point of indeterminacy} $(-2,-1,1)$. 
\begin{corollary}\label{newSP}
	Around $t=1$ we have
	\begin{align*}
		\zeta_{\textup{MT}}(-2,-1, t)=\frac{11}{1440}+O(|t-1|).
	\end{align*}
\end{corollary}
\begin{proof}
	Put $k=3$ in Theorem \ref{newFE} to see 
	\begin{align}\label{FEk3}
		F_{-3,0}\left(x\right)+\frac{1}{x^4} F_{-3,0}\left(\frac{1}{x}\right) = -\frac{1}{720x^4} \left( 5x^2 + 6\left(x^4 -1\right)\log(x) + (6\gamma - 720 \zeta'(-3))\left(x^4+1\right) \right).
	\end{align}
	On substituting \eqref{FEk3} in \eqref{klfk3}, one can get
	\begin{align}
		\Theta(-2,-1,t,x)&+ x^{1-t} \Theta\left(-2,-1, t,\frac{1}{x}\right)= \frac{3 + 5 x^2 + 3 x^4}{720x^2}+ O_{x}(|t-1|),
	\end{align}
	around $t=1$. 
	In particular, letting $x=1$ gives us the required special value.
\end{proof}
	\section{Modular relations: The Mordell--Tornheim zeta perspective}\label{SectionMTPer}
	In this section, we consider the class $F_{0,b}$ to make this study comprehensive. We alternatively derive the two-term identities of Ramanujan and Guinand as mentioned in the introduction. We then explain the underlying Mordell--Tornheim perspective and list the observations in Table \ref{table2}.  
	\subsection{Guinand's functional equations}
	\begin{theorem}
		Functional equations of Guinand given in \eqref{guinand2} and \eqref{guinand1} hold.
	\end{theorem}
	\begin{proof}
		From \cite[Equation 6.4.10, p.260]{Handbook}, for $j \in \mathbb{N}$, we can see,
		\begin{align}
			\psi^{(j)}(z) = (-1)^{j+1} j! \sum_{\ell=0}^{\infty} \frac{1}{(\ell+z)^{j+1}} = (-1)^{j+1} j! \zeta(j+1,z), \label{psi der}
		\end{align}
		where $\zeta(s,a)$ is the Hurwitz zeta function. From \eqref{Tinv}, we have
		\begin{align}
			\sum_{n=1}^{\infty} \sum_{m=1}^{\infty} \frac{1}{(n+mx)^{t}} = x^{-t} \sum_{n=1}^{\infty} \sum_{m=1}^{\infty} \frac{1}{(n+\tfrac{m}{x})^{t}}. \label{step2}
		\end{align}
		Replace $n$ by $n-1$ in the summands on both left-hand and right-hand side of the equation above and then use \eqref{psi der} to see that, for $t \in \mathbb{N}$ such that $t \geq 3$,
		\begin{align*}
			(-1)^{t}(t-1)!\sum_{m=1}^{\infty} \psi^{(t-1)}(1+mx)
			=(-1)^{t}(t-1)!x^{-t}\sum_{m=1}^{\infty} \psi^{(t-1)}\left(1+\frac{m}{x}\right).
		\end{align*}
		Simplify the above equation to get \eqref{guinand2}. We now prove the second identity of Guinand \eqref{guinand1}. From \eqref{Tinv}, we can also see that,
		\begin{align}
			\sum_{m=1}^{\infty} \zeta(t,mx+1) = x^{-t} 	\sum_{m=1}^{\infty} \zeta(t,\tfrac{m}{x}+1). \label{Hzeta}
		\end{align}
		From \cite[Equation 4.3]{Analogues}, we have,
		\begin{align*}
			\zeta(z,x) = \frac{x^{1-z}}{z-1} + O(x^{-z}).
		\end{align*}
		Adding and subtracting suitable terms from the summands on both left-hand and right-hand side of
		\eqref{Hzeta}, and simplifying, we get
		\begin{align}
			\sum_{m=1}^{\infty} \left( \zeta(t,mx+1)  - \frac{(mx)^{1-t}}{t-1} \right) + \frac{x^{1-t}\zeta(t-1)}{t-1} = x^{-t} \sum_{m=1}^{\infty} \left(\zeta(t,\tfrac{m}{x}+1) - \frac{\left(\frac{m}{x}\right)^{1-t}}{t-1} \right) + \frac{\zeta(t-1)}{x(t-1)}. \label{b4lim}
		\end{align}
		See the following limits hold, where we use \eqref{psi der} for the first one:
		\begin{align}
			\lim_{t \to 2} \left(  \zeta(t,y+1)  - \frac{y^{1-t}}{t-1}  \right) = \zeta(2,y+1) - \frac{1}{y} = \psi'(1+y) - \frac{1}{y}, \label{L1}
		\end{align}
		\begin{align}
			\lim_{t \to 2} \left( \left( x^{1-t} - \frac{1}{x} \right) \frac{\zeta(t-1)}{t-1} \right) = -\frac{\log(x)}{x} = - 2 \frac{\log(x)}{2x}= -  \frac{\log(x)}{2x} + \frac{\log(\frac{1}{x})}{2x} . \label{L2}
		\end{align}
		Tend $t \to 2$ in \eqref{b4lim}, use \eqref{L1} and \eqref{L2}, and simplify to get \eqref{guinand1}.
	\end{proof}
		\subsection{Ramanujan's functional equation}
		\begin{theorem}
		The first equality in the functional equation of Ramanujan given in \eqref{w1.26} holds.
	\end{theorem}
	\begin{proof}
		From \cite[Equation 4.3]{Analogues}, we have,
		\begin{align*}
			\zeta(z,x) = \frac{x^{1-z}}{z-1} + \frac{1}{2}x^{-z} + O(x^{-z-1}).
		\end{align*}
		Adding and subtracting suitable terms from the summands on both left-hand side and right-hand side of \eqref{Hzeta}, and then using $\zeta(z,a+1) = \zeta(z,a) - a^{-z}$, we get,
		\begin{align*}
			&\sum_{m=1}^{\infty} \left( \zeta(t,mx) + \frac{(mx)^{1-t}}{t-1} + \frac{1}{2}(mx)^{-t} - \frac{(mx)^{1-t}}{t-1} - \frac{1}{2}(mx)^{-t} - (mx)^{-t} \right) \\
			&= x^{-t} \sum_{m=1}^{\infty} \left(\zeta(t,\tfrac{m}{x}) + \frac{\left(\frac{m}{x}\right)^{1-t}}{t-1} + \frac{1}{2}\left(\frac{m}{x}\right)^{-t} - \frac{\left(\frac{m}{x}\right)^{1-t}}{t-1} - \frac{1}{2}\left(\frac{m}{x}\right)^{-t} - \left(\frac{m}{x}\right)^{-t} \right).
		\end{align*}
		Let us use the following notation for ease:
		\begin{align*}
			\varphi(t,x) := \zeta(t,x) - \frac{x^{1-t}}{t-1} - \frac{1}{2}x^{-t}.
		\end{align*}
		After simplification, we get,
		\begin{align*}
			\left(\sum_{m=1}^{\infty} \varphi(t,mx)  - \frac{\frac{1}{x}}{t-1} \zeta(t-1) - \frac{1}{2} x^{-t} \zeta(t) \right)= x^{-t} \left( \sum_{m=1}^{\infty} \varphi(t,\tfrac{m}{x})  - \frac{x}{t-1} \zeta(t-1) - \frac{1}{2} \left(\frac{1}{x}\right)^{-t} \zeta(t) \right).
		\end{align*}
		Multiply both sides by $x^{\frac{t}{2}}$ to get the first equality of \cite[Theorem 4.1]{Analogues}. As shown in \cite[Corollary 4.2]{Analogues}, letting $t \to 1$, we get the first equality of the Ramanujan's transformation formula \eqref{w1.26}.
	\end{proof}


	We list a few key observations in table \ref{table2}. This highlights how one can obtain various modular relations from the generalized Mordell--Tornheim zeta function. We start with the necessary arguments for $\Theta$ in the first column. After using \eqref{Tsplit} the number of times as in column 3, we do the operation  on $t$ as in column 2, to obtain the result mentioned in column 4. For example, $\Theta(2,2,t,x)$ can give  Vlasenko-Zagier's functional equations \eqref{vz2term} for different $r$, based on the steps mentioned in columns 2 and 3. Observe that $\Theta(0,0,t,x)$ acts as the center between Ramanujan's functional equation \eqref{w1.26}, Guinand's functional equations \eqref{guinand2} and \eqref{guinand1}, and a new functional equation \eqref{newfe}, differing by the limiting value of $t$ alone, bringing the four results under a single umbrella. 
	\renewcommand{\arraystretch}{1.4}
	\begin{table}[h!]
	\begin{center}
		\begin{tabular}{ | c| c| c| c| } 
			\hline
			$\Theta$& Operation on $t$& No. of splits&Result obtained  \\
			\hline
			\hline
			\multirow{4}{5em}{$\Theta(0,0,t,x)$} & $t \in \mathbb{N}$, $t\geq3$ &0& Guinand \eqref{guinand2} \\ 
			\cline{2-4}
			& $t \to 2$ &0& Guinand \eqref{guinand1} \\ 
			\cline{2-4}
			& $t \to 1$ &0& Ramanujan \eqref{w1.26} \\ 
			\cline{2-4}
			& $t \to 0$ &1& New \eqref{newfe} \\ 
			\hline
			\multirow{2}{5em}{$\Theta(1 ,1,t,x)$} 
			& $t \to 1$ &0& Vlasenko-Zagier \eqref{vz2term} with $r=2$ \\ \cline{2-4}
			&$t \to 0$&1& Zagier \eqref{fe2} \\ 
			\hline
			\multirow{3}{5em}{$\Theta(2,2,t,x)$} 
			&$t \to 1$&0& Vlasenko-Zagier \eqref{vz2term} with $r=4$ \\ \cline{2-4}
			&\multirow{1}{2.5em}{$t \to 0$}&1& Vlasenko-Zagier \eqref{vz2term} with $r=3$ \\ \cline{2-4}
			&$t \to -1$&2& Vlasenko-Zagier \eqref{vz2term} with $r=2$\\ [0.5ex]
			\hline
		\end{tabular}
		\caption{The table lists a few connections between the Mordell--Tornheim zeta function and the modular relations discussed above.}
	\label{table2}
	\end{center}
	\end{table}
	\section{Concluding Remarks}
\noindent We conclude this paper with the following remarks and questions:
	\begin{enumerate}
\item As explained in Section \ref{introduction}, it is very natural to ask for a generalization of \eqref{KLFzagier} and \cite[Theorem 2]{vz}, with the Hurwitz variable $z$. The connection between the generalized higher Kronecker ``limit'' formula and the Kronecker limit type formula for $\Theta(r,r,t,x)$ could then be investigated further.
\item As discussed in the Section \ref{SectionMTPer}, the Mordell--Tornheim zeta function gives a new perspective of the two-term functional equations in the literature. This interplay between $\Theta(r,r,t,x)$ and the modular relations thus widens the scope for the study on $\Theta(r,r,t,x)$, and more generally $\Theta(r,s,t,x)$. For example, having higher order terms of $\Theta(r,r,t,x)$ in \ref{KLF-r}, will give us two-term explicit functional equations Ishibashi and Higher Ishibashi functions, showed in \cite[Theorem 3.4]{dss2}.
\item  We have obtained two term functional equations as corollaries of Theorem \eqref{molty}. Three-term analogue of \eqref{molty}  will be an extremely desirable result. This is because, such an identity will enable us to derive three-term functional equations such as \eqref{fe1} and \eqref{vz3term} and put them in the perspective. 
\item For any fixed $a, b \in \mathbb{N}$, it is clear that $F_{a,b}(1) \in \textbf{Z}$. It will be interesting to know if it holds for other values of $a$ and $b$. Specifically, if $F(1)=-\frac{1}{2}\gamma^2-\frac{\pi^2}{12}-\gamma_1$ lies in $\textbf{Z}_2$. Furthermore, can something be said about $F_{a,b}(2)$, or the other special values?
\item As explained in \cite[Theorem 1.32]{GilFresan}, there is an $\mathbb{Q}$-\textit{algebra} structure on $\textbf{Z}$. Analogously, it would be interesting  to know if there is an $\mathbb{Q}((x))$-\textit{algebra} structure on the space of HZ type functions $\textbf{F}$, i.e., if one can define a `suitable' product operation in the space of HZ type functions. 
\item Lastly, to build on the space of HZ type functions, one might begin by looking at $F_{a,b}$ for non-integer values of $a$.  
\end{enumerate}

\section{Acknowledgment}
The authors would like to acknowledge Atul Dixit for his support. The first author is supported by the Seed Money Grant CUK/ACAD-II/F-3787/26 of Central University of Karnataka and Prime Minister Early Career Research Grant (PM-ECRG)  ANRF/ECRG/2025/004852/PMS of Anusandhan National Research Foundation (ANRF), Government of India. The second author thanks the National Board of Higher Mathematics for the NBHM Post Doctoral Fellowship (File no. 0204/21(18)/2025-R\&D-II/16314) of the Government of India. \\

\noindent \textbf{Data Availability Statement:} The manuscript has no associated data.

\noindent \textbf{Conflict of Interest Statement:} On behalf of all authors, the corresponding author states that there is no conflict of interest.

\end{document}